\documentclass{amsart}
\usepackage[margin=1.01 in]{geometry}

\usepackage{ifxetex,ifluatex}
\if\ifxetex T\else\ifluatex T\else F\fi\fi T%
  \usepackage{fontspec}
\else
  \usepackage[T1]{fontenc}
  \usepackage[utf8]{inputenc}
  \usepackage{lmodern}
  \usepackage{xcolor}
\fi

\usepackage{hyperref}
\usepackage{amsmath}
\usepackage{amssymb}
\usepackage{amsthm}
\usepackage{amssymb}
\usepackage{amsthm}
\usepackage{tikz-cd}
\usetikzlibrary{arrows.meta, positioning}
\usepackage{graphicx}

\DeclareMathOperator{\Gal}{Gal}

\DeclareMathOperator{\Aut}{Aut}

\newcommand{\cO}{ {\mathcal O} }

\newcommand{\bZ} { {\mathbb Z}}

\newcommand{\bQ} { {\mathbb Q}}

\newtheorem{theorem}{Theorem}
\newtheorem{lemma}[theorem]{Lemma}
\newtheorem{proposition}[theorem]{Proposition}
\newtheorem{corollary}[theorem]{Corollary}
\theoremstyle{definition}
\newtheorem{example}[theorem]{Example}
\theoremstyle{remark}
\newtheorem{remark}[theorem]{Remark}

\title[Quadratic Polynomials over $\bQ$ with Surjective Arboreal Galois Representations]{Examples of Quadratic Polynomials over $\bQ$ with Surjective Arboreal Galois Representations}
\author{Luck Henderson}
\author{Jamie Juul}
\author{Brenner Lattin}
\author{Enrique Mercado}
\author{Mia Schaefer}

\begin{document}

\begin{abstract}
We explore families of pairs of quadratic polynomials $f(x)=x^2+c\in \mathbb{Q}$ and $a\in \mathbb{Q}$ with $a$ being a strictly preperiodic point of $f$ to provide infinitely many new examples for which the associated arboreal Galois representations are surjective.
\end{abstract}

\maketitle

\section{Introduction}

Let $f(x)\in \bQ[x]$ and let $a\in \bQ$. Let $\bar{\bQ}$ denote an algebraic closure of $\bQ$ and let \[\cO^-(a) = \{\beta\in \bar{\bQ}: f^n(\beta)=a \text{ for some } n\geq 0\}\] denote the \textit{backward orbit} of $a$ under the map $f$. This backward orbit has a natural structure as a regular $d$-ary rooted tree graph whenever $a$ is not in the forward orbit of a critical point. This tree structure is given by assigning elements of $f^{-n}(a)$ to nodes at the $n$-th level of the tree with an edge between $\beta$ and $f(\beta)$. The extension $\bQ(\cO^-(a))/\bQ$ is a Galois extension and the absolute Galois group $\Gal(\bar{\bQ}/\bQ)$ acts on the backward orbit in a way that preserves the tree structure (see Figure~\ref{fig:tree}. Thus, we get a group homomorphism 
\[\Gal(\bar{\bQ}/\bQ)\rightarrow \Aut(T),\] 
where $T$ denotes the regular $d$-ary rooted tree graph and $\Aut(T)$, its automorphism group. This homomorphism is called an \textit{arboreal Galois representation}. The image of this representation is isomorphic to $\Gal(\bQ(\cO^-(a))/\bQ)$. These Galois representations are conjectured to have images with finite index in $\Aut(T)$ except in very specific cases \cite[Conjecture 3.11]{Jones_survey}, \cite[Question 1.1]{Bridyetal}. See \cite{Jones_survey} for a more detailed introduction to arboreal Galois representations. 

\begin{figure}[h]
% https://tikzcd.yichuanshen.de/#N4Igdg9gJgpgziAXAbVABwnAlgFyxMJZAZgBpiBdUkANwEMAbAVxiRAB12AjJhhmHCAC+pdJlz5CKACykAjFVqMWbTjz4DhokBmx4CRAEzzF9Zq0QduvfoJFi9kogAYT1Mystqbm+zvH6UsgAbG5K5qrWGnbauhIGKACspIamyhZW6rZaDvFBcilpEV5R2X5xgUQA7ORFniB0Of6OCcg1qe7pbABmAHrAALRyQgAUdACUTRVOKDUKncUgfYOGoxNTATMkpM51GZw0UBA4CEKKMFAA5vBEoN0AThAAtkhkIDgQSKEgABYwdFA2JAwKw-A9nkhjO9Pohvn8AUCCKDtOCXogCtCkMlfv9AZZgci7o80bJMYhsfC8eAkU1UV9qB8kK4cQj8TSwcSsQyYczKYiQcIKEIgA
\begin{tikzcd}[sep=small]
                            &                              &                             & \vdots  &                             &                              &                             &           \\
\bullet \arrow[rd, no head] &                              & \bullet \arrow[ld, no head] &         & \bullet \arrow[rd, no head] &                              & \bullet \arrow[ld, no head] & f^{-2}(a) \\
                            & \bullet \arrow[rrd, no head] &                             &         &                             & \bullet \arrow[lld, no head] &                             & f^{-1}(a) \\
                            &                              &                             & \bullet &                             &                              &                             & a        
\end{tikzcd}
\caption{Tree diagram for backward orbit of $a$.}
\label{fig:tree}
\end{figure}
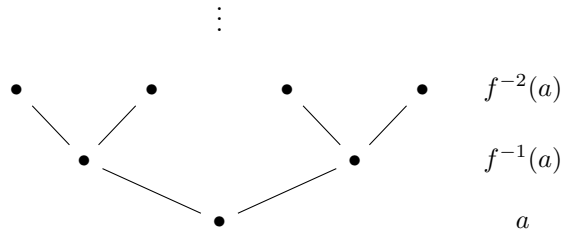

Several sets of sufficient conditions under which these arboreal Galois representations are surjective, especially in the quadratic case, exist in the literature (see for example \cite{odoni}, \cite{Stoll}, \cite{Jones_thedensity} for quadratic examples and \cite{Juuletal_19WIN4}, \cite{Looper}, \cite{Kadets}, \cite{BenedettoJuul_19odoni}, \cite{Specter} for degree 2 and higher degree examples). 
Adding to this body of literature, we prove the following two theorems.

\begin{theorem}\label{thm:main1}
Let $a\in \mathbb{Q}$, let $c=-a-a^2$, and let $f(x)=x^2+c$. Write $a=\frac{r}{s}$ where $r,s\in \mathbb{Z}$, $s>0$, and $\gcd(r,s)=1$.
Define 
\begin{equation*}
\delta = 
\begin{cases}
0 & \text{if } a\in (-\infty, -2)\cup (-2,-1)\cup(-1,0)\cup (1,\infty)\\
1 & \text{if } a\in (0,\beta)
\end{cases}
\end{equation*}
where $\beta =\frac{1}{3} (-2 + (19 - 3\sqrt{33})^\frac{1}{3} + (19 + 3\sqrt{33})^\frac{1}{3}) \approx 0.839$.
Also define
\begin{equation*}
e=\begin{cases}
0 & \text{if } v_2(a)\leq 0\\
1 & \text{if } v_2(a)>0
\end{cases}.
\end{equation*}
Suppose $a-c$ is not a square in $\mathbb{Q}$ and at least one of the following hold:
\begin{enumerate}
\item\label{main1_cond1} $(-1)^\delta 2^e |r|\equiv 2\mod 3$
\item\label{main1_cond2} $(-1)^\delta 2^e |r|\equiv 3\mod 4$
\item\label{main1_cond3} $(-1)^\delta 2^e |r|$ is not a quadratic residue modulo $q$ for some prime $q$ dividing $s$.
\end{enumerate}
Then the associated arboreal Galois representation $\Gal(\bar{\mathbb{Q}}/\mathbb{Q})\rightarrow \Aut(T)$ is surjective.
\end{theorem}

\begin{theorem}\label{thm:main2}
Let $a\in \mathbb{Q}$, let $c=-1+a-a^2$, and let $f(x)=x^2+c$. Write $a=\frac{r}{s}$ where $r,s\in \mathbb{Z}$, $s>0$, and $\gcd(r,s)=1$. 
Suppose $a-c$ is not a square in $\mathbb{Q}$ and at least one of the following hold:
\begin{enumerate}
\item\label{main2_cond1} $r=1$ and $s>2$ is even,
\item\label{main2_cond2} $r=2$, $s>3$, and $s\equiv 1\mod 3$,
\item\label{main2_cond3} $r=2$ and $s$ is divisible by a prime $q$ with $q\equiv 3\mod 4$.
\end{enumerate}
Then the associated arboreal Galois representation $\Gal(\bar{\mathbb{Q}}/\mathbb{Q})\rightarrow \Aut(T)$ is surjective.
\end{theorem}

Note, we work with parameters $c$ in $\mathbb{Q}$, rather than $\mathbb{Z}$, where much of the previous work in this direction has been focused. This allows us to generate many examples with parameters $c$ with small absolute value and/or small height. For example, we can construct infinitely many examples with $-2\leq c\leq \frac{1}{4}$, which means $c$ will be in the Mandelbrot set, so the Julia set of the map (which is approximated by a random backward orbit) will be connected (see Figure~\ref{fig:julia} below). 

\begin{example} We give several examples where these theorems imply $\Gal(\bar{\mathbb{Q}}/\mathbb{Q})\rightarrow \Aut(T)$ is surjective.
\begin{itemize}
\item Let $a=\frac{1}{5}$ and let $c= -\frac{6}{25}$. Then $a-c=\frac{11}{25}$ and $(-1)^1\cdot 2^0\cdot |1|\equiv 2\mod 3$, so Theorem~\ref{thm:main1}, condition (\ref{main1_cond1}) is met.
\item Let $a=\frac{1}{2}$ and let $c=-a^2-a = -\frac{3}{4}$. Then $a-c = \frac{5}{4}$ and $(-1)^1 \cdot 2^0\cdot |1|\equiv 3\mod 4$, so Theorem~\ref{thm:main1}, condition (\ref{main1_cond2}) is met.
\item Let $a=-\frac{6}{7}$ and let $c=-a^2-a=\frac{6}{49}$. Then $a-c = -\frac{48}{49}$ and $(-1)^0\cdot 2^1\cdot |6|\equiv 5\mod 7$, so Theorem~\ref{thm:main1},  condition (\ref{main1_cond3}) is met.
\item Let $a=\frac{1}{4}$ and let $c=-\frac{13}{16}$. Then $a-c=\frac{17}{16}$ and Theorem~\ref{thm:main2},  condition (\ref{main2_cond1}) is met.
\item Let $a=\frac{2}{13}$ and let $c=-\frac{147}{169}$. Then $a-c=\frac{173}{169}$ and Theorem~\ref{thm:main2},  condition (\ref{main2_cond2}) is met
\item Let $a=\frac{2}{3}$ and let $c=-\frac{7}{9}$. Then $a-c=\frac{13}{9}$ and Theorem~\ref{thm:main2},  condition (\ref{main2_cond3}) is met.
\end{itemize}
\end{example}

\begin{figure}[h]
\includegraphics[scale=.4]{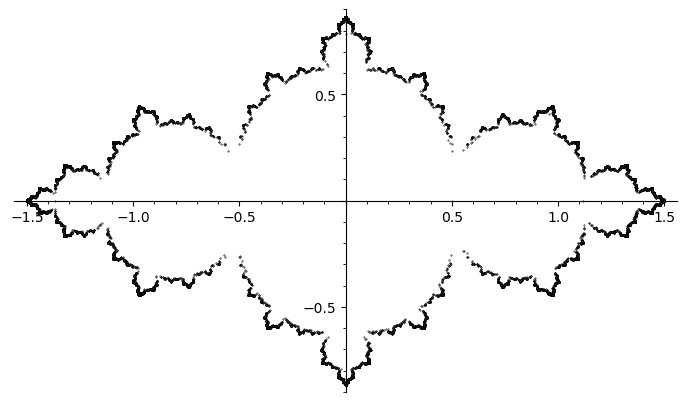}
\caption{Julia set of $f(x)=x^2-\frac{3}{4}$ approximated by backward orbit of $a=\frac{1}{2}$.}
\label{fig:julia}
\end{figure}

Our proof relies on a well known argument, originally due to Odoni \cite{odoni, odoni_realizing} and extended by Stoll \cite{Stoll}, Jones \cite{Jones_survey}, and others showing the arboreal Galois representation for a quadratic polynomial $f(x)=x^2+c$ is surjective if and only if the sequence $a-c,f(0)-a,f^2(0)-a, f^3(0)-a,\dots$ is $2$-independent. We state this result as Lemma~\ref{lemma:surjective} and refer the reader to \cite{BGJT_25PCFper} for a full proof or one of the papers cited above for various versions of the argument.
Inspired by a strategy of Odoni \cite{odoni} and Jones \cite{Jones_thedensity}, to generate examples satisfying the hypotheses of Lemma~\ref{lemma:surjective} we construct our families so that the base point $a$ is strictly preperiodic, which allows us to control the repeated prime factors of the terms of sequence $f^n(0)-a$. We construct the families in Section~\ref{sec:families} and study the prime divisors of the sequence in Section~\ref{sec:repeatedprimes}. We then analyze when the elements of this sequence can be shown to be positive or negative in Section~\ref{sec:signs}. Finally, in Section~\ref{sec:proofs}, we put the work of the previous sections together to prove Theorem~\ref{thm:main1} and Theorem~\ref{thm:main2}.

\subsection*{Acknowledgements} The authors would like to thank Paul Fili for asking the question that led to this project and for the Sage code that was used to generate Figure~\ref{fig:julia}.

\section{Families of quadratic polynomials with preperiodic points over $\mathbb{Q}$}\label{sec:families}

In this section, we parametrize pairs $f(x)=x^2+c$, $a$ where $a$ is a strictly preperiodic point of $f$ with tail length 1 and cycle length 1 or 2. 

\begin{proposition}
Let $f(x)=x^2+c$ and let $a$ be a strictly preperiodic point with tail length 1 and cycle length 1. Then \[c=-a-a^2\] with $a\neq 0$.
\end{proposition}

\begin{proof}
Since the orbit of $a$ has tail length 1 and cycle length 1, we have $f^2(a)=f(a)$, that is,
\[(a^2+c)^2+c=a^2+c.\]
The solutions to this equation are
\[c=\begin{cases} -a-a^2 \\ a-a^2 \end{cases}.\]
When $c=a-a^2$, we have $f(a)=a$, but we have assumed $a$ is not periodic. When $c=-a-a^2$, we have $f(a)=-a$ and $f^2(a)=-a$, hence if $a\neq 0$,
then $a$ is strictly preperiodic with the desired orbit (shown in Figure~\ref{fig:cycle_length_1}).
\end{proof}

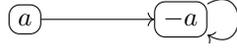
\begin{figure}[h]
\begin{tikzpicture}[
    node distance=1.5cm,
    every node/.style={draw, rectangle, rounded corners, align=center}
]
\node (a) {$a$};
\node (neg_a) [right=of a] {$-a$};
\draw[->] (a) to (neg_a);
\draw[->] (neg_a) to [looseness=4, out=30,  in=-30] (neg_a);
\end{tikzpicture}
\caption{Orbit of $a$ for $f(x)=x^2+(-a-a^2)$.}
\label{fig:cycle_length_1}
\end{figure}

We now consider the case where the orbit of $a$ has tail length 1 and cycle length 2.

\begin{proposition}
Let $f(x)=x^2+c$ and let $a$ be a strictly preperiodic point with tail length 1 and cycle length 2. Then \[c=-1+a-a^2\] with $a\neq 0, \frac{1}{2}$.
\end{proposition}

\begin{proof}
Since the orbit of $a$ has tail length 1 and orbit length 2, we have $f^3(a)=f(a)$, so 
\[((a^2+c)^2+c)^2+c=a^2+c\]
This has solutions
\[c=\begin{cases}
-1-a-a^2\\
-1+a-a^2\\
-a-a^2\\
a-a^2.
\end{cases}\]
The first and last solutions above make $a$ periodic (with period 2 and period 1 respectively), while we have seen the third solution corresponds to $a$ strictly periodic with tail length 1 and cycle length 1. When $c=-1+a-a^2$, we have $f(a)=a-1$, $f^2(a)=-a$, and $f^3(a)=a-1=f(a)$, which gives the desired orbit (shown in Figure~\ref{fig:cycle_length_2}) as long as $a\neq 0$ (so $a\neq -a$) and $a\neq \frac{1}{2}$ (so $-a\neq a-1$).
\end{proof}

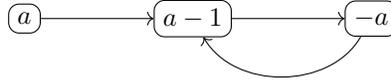
\begin{figure}[h]
\begin{tikzpicture}[
    node distance=1.5cm,
    every node/.style={draw, rectangle, rounded corners, align=center}
]

\node (a) {$a$};
\node (a1) [right=of a] {$a-1$};
\node (neg_a) [right=of a1] {$-a$};

\draw[->] (a) to (a1);
\draw[->] (a1) to (neg_a);
\draw[->] (neg_a) to [out=-120,  in=-60] (a1);

\end{tikzpicture}
\caption{Orbit of $a$ for $f(x)=x^2+(-1+a-a^2)$.}
\label{fig:cycle_length_2}
\end{figure}

Rational preperiodic points of rational quadratic maps have been studied extensively prior to this paper. In \cite{Poonen_classification}, Poonen classifies the quadratic polynomials $x^2+c\in \mathbb{Q}[x]$ which have rational periodic points. He shows $x^2+c\in \mathbb{Q}[x]$ has a rational fixed point if and only if $c=\frac{1}{4}-\rho^2$ for some $\rho\in \mathbb{Q}$ and in this case, it has $2$ rational fixed points $\frac{1}{2}-\rho$ and $\frac{1}{2}+\rho$ and the strictly periodic points mapping into these cycles with tail length 1 are $-(\frac{1}{2}-\rho)$ and $-(\frac{1}{2}+\rho)$ respectively. 
Comparing this classification with our first family, taking $a=-(\frac{1}{2}+\rho)$ and $c=\frac{1}{4}-\rho^2=-a^2-a$ produces our family of strictly preperiodic point, polynomial pairs with the specified orbit. The other strictly preperiodic point for this map is produced by choosing $a'=-1-a=-(\frac{1}{2}+\rho)$.  

Poonen also shows $x^2+c\in \mathbb{Q}[x]$ has a rational point of period 2 if and only if $c=-\frac{3}{4}+\sigma^2$ for some $\sigma\in \mathbb{Q}\setminus \{0\}$, and in this case there are exactly two $-\frac{1}{2}+\sigma$ and $-\frac{1}{2}-\sigma$, which have strictly preperiodic preimages $\frac{1}{2}-\sigma$ and $\frac{1}{2}+\sigma$ respectively. We recover our classification of point, polynomial pairs by taking $a$ to be either of these, choosing $a'=1-a$ produces the other point, polynomial pair with the same polynomial.

Note, Poonen also classifies $x^2+c\in \mathbb{Q}[x]$ which has a rational point of period 3. We omit this case, but it would be another interesting source of examples. Conjecturally, a quadratic polynomial with coefficients in $\mathbb{Q}$ has no points of period greater than $3$ \cite{FlynnPoonenSchaefer}.
Moreover, Poonen classifies when there will be a strictly preperiodic point with tail length greater than one. For our families, this corresponds to when $a-c$ is a square in $\mathbb{Q}$. In these cases the Galois representation will not be surjective, but one could consider the backward orbit of the preimages of $a$ to look for further examples. 

\section{Repeated prime divisors in adjusted critical orbits}\label{sec:repeatedprimes}

We first show, for primes outside a finite set, if a prime $p$ divides the numerator of two terms in the sequence $\{f^n(0)-a\}_{n\geq 1}$, then it must divide the numerator of $f^n(a)-a$ for some $n\geq 1$. We then apply this result to understand the prime divisors of this sequence for the families introduced in Section~\ref{sec:families}. Let $\bZ_{(p)}$ denote the localization of $\bZ$ at $(p)$.

\begin{lemma}\label{lem:common_primes} Let $p\in \bZ$ be prime. Suppose $f(x)\in \bZ_{(p)}[x]$ and $a\in \bZ_{(p)}$. If $v_p(f^m(0)-a), v_p(f^n(0)-a)>0$ for $m,n\geq 1$ with $m\neq n$, then $v_p(f^\ell(a)-a)>0$ for some $\ell\geq 1$.
\end{lemma}

\begin{proof}
Suppose $v_p(f^m(0)-a), v_p(f^n(0)-a)>0$ for $1\leq m < n$. Write $n = m+\ell$ for some $\ell \geq 1$. Then in the residue field $\bZ_{(p)}/(p)$, we have $\bar{f}^m(\bar{0})-\bar{a}=\bar{f}^{m+\ell}(\bar{0})-\bar{a}=\bar{0}$. Thus, $\bar{f}^m(\bar{0})=\bar{a}$ and $\bar{0}=\bar{f}^{m+\ell}(\bar{0})-\bar{a}=\bar{f}^\ell(\bar{f}^m(\bar{0}))-\bar{a}=\bar{f}^\ell(\bar{a})-\bar{a}$. Therefore, $v_p(f^\ell(a)-a)>0$.
\end{proof}

Using Lemma~\ref{lem:common_primes}, we show for a polynomial of the form $f(x)=x^2-a-a^2$, the only primes that divide the numerator of $f^n(0)-a$ for more than one $n$ are those dividing the numerator of $a$. We also analyze the $p$-adic valuation of the numerator of $f^n(0)-a$ for these primes.

\begin{proposition}\label{prop:family1primes}
Fix a polynomial in the family $f(x)=x^2+c$ with  $c=-a-a^2$. Let $p\in \mathbb{Z}$ be prime. 
\begin{enumerate}
\item\label{item:primesindenom} The following are equivalent:
  \begin{itemize}
    \item $v_p(a)<0$,
    \item $v_p(f^n(0)-a)<0$ for some $n\geq 1$,
    \item $v_p(f^n(0)-a)<0$ for all $n\geq 1$.
   \end{itemize}
\item\label{item:aodd} If $v_2(a)=0$, then $v_2(f^n(0)-a)=0$ for all $n$.
\item\label{item:aeven} If $v_2(a)\geq 1$, then $v_2(f^n(0)-a)=v_2(a)+1$ for all $n\geq 2$.
\item\label{item:primesinnum} If $v_p(a)>0$ and $p\neq 2$, then $v_p(f^n(0)-a)=v_p(a)$ for all $n\geq 1$.
\item\label{item:repeatedprimes} If $v_p(f^m(0)-a), v_p(f^n(0)-a)>0$ for $m,n\geq 1$ with $m\neq n$, then $v_p(a)>0$.
\end{enumerate}
\end{proposition}

\begin{proof}
Write $a=\frac{r}{s}$ with $\gcd(r,s)=1$ and $s>0$. Similarly, for $n\geq 0$, write $f^n(0)-a=\frac{r_n}{s_n}$ with $\gcd(r_n,s_n)=1$ and $s_n>0$. Note $r_0=-r$ and $s_0=s$. 

We first claim $s_n=s^{2^n}$. We proceed by induction on $n$. The claim holds for $n=0$. Suppose $s_n=s^{2^n}$ for some $n\geq 0$. Then 
\begin{align*}
\frac{r_{n+1}}{s_{n+1}} &= f^{n+1}(0)-a\\
&=f(f^n(0)-a+a)-a\\
&=f\left(\frac{r_n}{s^{2^n}}+\frac{r}{s}\right)-\frac{r}{s}\\
&=\left(\frac{r_n}{s^{2^n}}+\frac{r}{s}\right)^2-\frac{r^2}{s^2}-\frac{2r}{s}\\
&=\frac{r_n^2+2r_n r s^{2^{n}-1}-2rs^{2^{n+1}-1}}{s^{2^{n+1}}}.
\end{align*}
Since $\gcd(r_n, s^{2n})=1$ by hypothesis, we can see \[\gcd(r_n^2+2r_n r s^{2^n-1}-2rs^{2^{n+1}-1}, s)=1,\] hence, \[r_{n+1}=r_n^2+2r_n r s^{2^n-1}-2rs^{2^{n+1}-1}\] and $s_{n+1}=s^{2^{n+1}}$.
Therefore (\ref{item:primesindenom}) holds.

If $v_2(a)=0$, then $v_2(r_0)=v_2(r)=0$. By the formula  $r_{n+1}=r_n^2+2r_n r s^{2^n-1}-2rs^{2^{n+1}-1}$ and induction on $n$, (\ref{item:aodd}) holds.

We now consider the $2$-adic valuation of $f^n(0)-a$ when $v_2(a)>0$. If $v_2(a)=v_2(r)=1$, then \[v_2(r_1)=v_2(-r^2-2rs)\geq \min\{v_2(r^2),v_2(2rs)\}=2= 1+v_2(r).\] If $v_2(a)=v_2(r)\geq 2$, then \[v_2(r_1)=v_2(-r^2-2rs)=v_2(2rs)=1+v_2(r),\] since $v_2(r^2)=2v_2(r)>1+v_2(r)$. We now prove (\ref{item:aeven})  by induction on $n$, suppose $v_2(r_n)\geq 1+v_2(r)$ for some $n\geq 1$. Then 
\[v_2(r_{n+1})=v_2(r_n^2+2r_n r s^{2^n-1}-2rs^{2^{n+1}-1})=v_2(2rs^{2^{n+1}-1})=1+v_2(r),\] since $v_2(r_n^2)=2v_2(r_n)$ and $v_2(2r_nrs^{2^{n}-1})=1+v_2(r_n)+v_2(r)$, which are both strictly greater than $1+v_2(r)$.

Now suppose $p\neq 2$ and $v_p(r_n)=v_p(r)>0$ for some $n\geq 0$. Then
\[v_p(r_{n+1})=v_p(r_n^2+2r_n r s^{2^n-1}-2rs^{2^{n+1}-1})= v_p(2rs^{2^{n+1}-1})=v_p(r)\]
since $v_p(r_n^2)=v_p(2r_n r s^{2^n-1})=2v_p(r)>v_p(r)$. Hence (\ref{item:primesinnum}) holds by induction.

Finally, suppose $v_p(f^m(0)-a), v_p(f^n(0)-a)>0$ for $m\neq n$. Note, $p\nmid s$ by (\ref{item:primesindenom}), so $f(x)\in \bZ_{(p)}$ and $a\in \bZ_{(p)}$. Hence by Lemma~\ref{lem:common_primes}, $v_p(f^\ell(a)-a)>0$ for some $n$. We constructed $f$ so that $f^\ell(a) = -a$ for all $\ell>0$, and hence $f^\ell(a)-a = -2a$ for all $\ell>0$. So we have $v_p(-2a)>0$.  
It follows immediately that if $p$ is odd, then $v_p(a)>0$. On the other hand, if $p=2$, then since $v_2(f^m(0)-a)>0$, it follows from (\ref{item:aodd}) and (\ref{item:primesindenom}) that $v_2(a)>0$. Therefore, (\ref{item:repeatedprimes}) holds.
\end{proof}

\begin{corollary}\label{cor:r_nfam1} With notation as in Proposition~\ref{prop:family1primes}, we can write $|r_1|=2^{e_1}|r|t_1$ for some $e_1,t_1\in \mathbb{Z}$ with $e_1\geq 0$ and $t_1>0$ and we can write $|r_n|=2^e|r|t_n$ where 
\begin{align*}
e=\begin{cases}
0 & \text{if } v_2(a)=0\\
1 & \text{if } v_2(a)\geq 1
\end{cases},
\end{align*} $t_n\in \mathbb{Z}$, and $t_n>0$. Further, for all $i\geq 1$, $2\nmid t_i$, $\gcd(t_i,r)=1$, and $\gcd(t_i,t_j)=1$ for all $i>j\geq 1$.
\end{corollary}

In the next proposition, we show for a polynomial of the form $f(x)=x^2-1+a-a^2$, the only primes that divide the numerator of $f^n(0)-a$ for more than one $n$ are those dividing the numerator of $2a$. For this family, controlling the $p$-adic valuation of repeated prime factors is difficult, except for the case of controlling the $2$-adic valuation if $v_2(a)=1$ or $0$.

\begin{proposition}\label{prop:family2primes} Fix a polynomial in the family $f(x)=x^2+c$ with $c=-1+a-a^2$. Let $p\in \mathbb{Z}$ be prime. 
\begin{enumerate}
\item\label{item:primesindenom_fam2}  The following are equivalent:
  \begin{itemize}
    \item $v_p(a)<0$,
    \item $v_p(f^n(0)-a)<0$ for some $n\geq 1$,
    \item $v_p(f^n(0)-a)<0$ for all $n\geq 1$.
   \end{itemize}
\item\label{item:2innum_fam2} If $v_2(a)>0$, then $v_2(f^n(0)-a)>0$ for all even $n\geq 2$ and $v_2(f^n(0)-a)=0$ for odd $n\geq 1$.
\item\label{item:a2mod4_fam2} If $v_2(a)=1$, then $v_2(f^n(0)-a)=2$ for all even $n\geq 2$ and $v_2(f^n(0)-a)=0$ for odd $n\geq 1$.
\item\label{item:aodd_fam2} If $v_2(a)=0$, then $v_2(f^n(0)-a)=1$ for all odd $n\geq 1$ and $v_2(f^n(0)-a)=0$ for even $n\geq 0$.
\item\label{item:repeatedprimes_fam2} If $v_p(f^m(0)-a), v_p(f^n(0)-a)>0$ for $m,n\geq 1$ with $m\neq n$, then $v_p(2a)>0$.
\end{enumerate}
\end{proposition}

\begin{proof}
Write $a=\frac{r}{s}$ with $r,s\in \mathbb{Z}$, $s>0$, and $\gcd(r,s)=1$. Similarly, write $f^n(0)-a=\frac{r_n}{s_n}$ with $r_n,s_n\in \mathbb{Z}$, $s_n>0$, and $\gcd(r_n,s_n)=1$.

We first show $s_n=s^{2^n}$ by induction on $n$. Note, this holds for $n=0$. Suppose $s_n=s^{2^n}$. Then 
\begin{align*}
\frac{r_{n+1}}{s_{n+1}}&=f^{n+1}(0)-a\\
&=f(f^n(0)-a+a)-a\\
&=f\left(\frac{r_n}{s^{2^n}}+\frac{r}{s}\right)-\frac{r}{s}\\
&=\left(\frac{r_n}{s^{2^n}}+\frac{r}{s}\right)^2-\frac{r^2}{s^2}+\frac{r}{s}-1-\frac{r}{s}\\
&=\frac{r_n^2+2r_nrs^{2^n-1}-s^{2^{n+1}}}{s^{2^{n+1}}}.
\end{align*}
Our induction hypothesis implies $\gcd(r_n,s)=1$, hence $\gcd(r_n^2+2r_nrs^{2^n-1}-s^{2^{n+1}},s^{2^{n+1}})=1$ and we have $s_{n+1}=s^{2^{n+1}}$. Thus, (\ref{item:primesindenom_fam2}) holds. We also get the recursive formula \[r_{n+1}=r_n^2+2r_nrs^{2^n-1}-s^{2^{n+1}}\] for the numerator of $f^{n+1}(0)-a$.

Now consider the $2$-adic valuation of $f^n(0)-a$. Suppose $v_2(a)>0$, so we have $r_0=-r$ is even and $s$ is odd. Then $r_1=-r^2-s^2$ is odd. Continuing by induction on $n$, if $r_n$ is odd, then $r_{n+1}=r_n^2+2r_nrs^{2^n-1}-s^{2^{n+1}}$ is even, and if $r_n$ is even then $r_{n+1}$ is odd, proving (\ref{item:2innum_fam2}). Further, if $v_2(a)=1$ and $n$ is even, then $r_{n-1}$ is odd as noted above and 
\[v_2(r_{n})=v_2(r_{n-1}^2+2r_{n-1}rs^{2^{n-1}-1}-s^{2^{n}})=v_2(2r_{n-1}rs^{2^{n-1}-1})=2,\] 
since $v_2(r_{n-1}^2-s^{2^{n}})\geq 3$, as the difference of two odd squares is divisible by $8$. Thus, (\ref{item:a2mod4_fam2}) holds. 

On the other hand, if $v_2(a)=0$, then $r$ and $s$ are both odd. So we have $r_0=-r$ is odd and $r_1=-r^2-s^2$ is even. The inductive argument above proves $r_n$ will alternate between even and odd values. Further, if $n$ is odd, we have $r_{n-1}$ is odd and so
\[v_2(r_{n})=v_2(r_{n-1}^2+2r_{n-1}rs^{2^{n-1}-1}-s^{2^{n}})=v_2(2r_{n-1}rs^{2^{n-1}-1})=1,\] as again, $v_2(r_{n-1}^2-s^{2^{n}})\geq 3$. Thus, (\ref{item:aodd_fam2}) holds.

Finally, suppose $v_p(f^m(0)-a), v_p(f^n(0)-a)>0$ for $m\neq n$. Note, $p\nmid s$ by (\ref{item:primesindenom_fam2}), so $f(x)\in \bZ_{(p)}$ and $a\in \bZ_{(p)}$. Hence by Lemma~\ref{lem:common_primes}, $v_p(f^\ell(a)-a)>0$ for some $\ell\geq 1$. We have $f^\ell(a) = -a$ or $a-1$ for each $\ell$, so $f^\ell(a)-a = -2a$ or $-1$, proving (\ref{item:repeatedprimes_fam2}).
\end{proof}

\section{Evaluating the signs of elements in the adjusted critical orbits}\label{sec:signs}

\begin{proposition}\label{prop:family1signs}
Fix a polynomial $f(x)=x^2+c$ with $c=-a-a^2$. 
\begin{enumerate}
\item If $a\in (-2,0)$, then $f^n(0)-a>0$ for all $n\geq 1$.
\item If $a\in (-\infty, -2]\cup[1,\infty)$, then $f^n(0)-a>0$ for all $n\geq 2$.
\item If $a\in (0, \beta)$, then $f^n(0)-a<0$ for all $n\geq 1$.
\end{enumerate}
where $\beta =\frac{1}{3} (-2 + (19 - 3\sqrt{33})^\frac{1}{3} + (19 + 3\sqrt{33})^\frac{1}{3}) \approx 0.839$ is the non-zero real root of $f^2(0)-a$.
\end{proposition}

\begin{proof}
First consider $a\in(-2,0)$. Then $c-a=-a^2-2a=-a(a+2)>0$, so $c>a$. Then for any $n \geq 1$, we have 
\[f^n(0)=f(f^{n-1}(0))=(f^{n-1}(0))^2+c\geq c>a,\] 
and hence $f^n(0)-a>0$.

Next, consider $a\in (-\infty,-2]\cup [1,\infty)$. We can compute  \[f^2(0)=a^4+2a^3-a>0\]  and 
\[f^3(0)-f^2(0)=a^3(a-1)(a+1)^3(a+2)\geq 0,\]  for all $a\in (-\infty,-2]\cup [1,\infty)$, so $f^3(0)\geq f^2(0)>0$. We can also check $f^2(0)-a>0$  for all $a\in (-\infty,-2]\cup [1,\infty)$. Now suppose we have $f^{n}(0)\geq f^{n-1}(0)>0$ for some $n\geq 3$. Then squaring both sides of $f^{n}(0)\geq f^{n-1}(0)$ and adding $c$, we have
\[f^{n+1}(0)=(f^{n}(0))^2+c\geq (f^{n-1}(0))^2+c=f^n(0)>0.\] 
Hence, $f^{n+1}(0)\geq f^n(0)>0$ for all $n\geq 2$, by induction on $n$. Using this and the fact that $f^2(0)>a$, we have \[\dots, f^{n+1}(0)\geq f^n(0)\geq \dots>f^2(0)> a.\] Therefore, $f^n(0)-a>0$ for all $n\geq 2$.

Now consider $a\in (0,\beta)$. First note \[f^2(0)-a=a^4+2a^3-2a<0\] on this interval. Next note $a< a^2+a=-c$. So we have \[f^2(0)<a< -c.\] We claim \[c\leq f^n(0)\leq f^2(0)\] for all $n\geq 1$. We have $f(0)=c\leq c^2+c=f^2(0)$, so the claim holds for $n=1$ and $n=2$. Now suppose $c\leq f^n(0)\leq f^2(0)$ for some $n\geq 2$. Then $c\leq f^n(0)\leq -c$, so \[0\leq (f^n(0))^2\leq c^2.\] Adding $c$ to each expression, we see \[c\leq (f^n(0))^2+c\leq c^2+c,\] that is, $c\leq f^{n+1}(0)\leq f^2(0)$ as desired. Now since $f^n(0)\leq f^2(0)<a$ we have $f^n(0)-a<0$ for each $n\geq 1$.

\end{proof}

\begin{remark} We note the following for completeness. If $a\in (\beta,1)$, the sequence $f^n(0)-a$ contains both positive and negative values.  If $a\in (-\infty,-2]\cup [1,\infty)$, then $f(0)-a=-a^2-2a<0$, except in the case $a=-2$, where we have $f(0)-a=0$.
\end{remark}

\begin{proposition}\label{prop:family2signs}
Fix a polynomial $f(x)=x^2+c$ with $c=-1+a-a^2$.
\begin{enumerate}
\item If $a\in \left(-\infty, \frac{1}{2}-\frac{\sqrt{5}}{2}\right)\cup\left(\frac{1}{2}+\frac{\sqrt{5}}{2}, \infty\right)$, then $f^n(0)-a>0$ for all $n\geq 2$.
\item If $a\in (0,\gamma)$, then $f^n(0)-a<0$ for all $n \geq 1$.
\end{enumerate}
where $\gamma = \frac{1}{3}\left(2-\frac{2}{(17+3*\sqrt{33})^{1/3}}+(17+3*\sqrt{33})^{1/3}\right) \approx 1.54$ is the nonzero real root of $f^2(0)-a$. 
\end{proposition}

\begin{proof} First suppose $a\in \left(-\infty, \frac{1}{2}-\frac{\sqrt{5}}{2}\right)\cup\left(\frac{1}{2}+\frac{\sqrt{5}}{2}, \infty\right)$. We can compute
\[f^2(0)=a^4-2a^3+2a^2-a>0\]
and
\[f^3(0)-f^2(0)=a^8 - 4 a^7 + 8 a^6 - 10 a^5 + 7 a^4 - 2 a^3 - 2 a^2 + 2 a - 1>0\]
for all $a\in \left(-\infty, \frac{1}{2}-\frac{\sqrt{5}}{2}\right)\cup\left(\frac{1}{2}+\frac{\sqrt{5}}{2}, \infty\right)$. Further, one can check $f^2(0)-a>0$ for all $a$ in these intervals. So we have $f^3(0)>f^2(0)>0$ and $f^2(0)>a$. Now suppose $f^{n}(0)>f^{n-1}(0)>0$ for some $n\geq 3$. Then squaring both sides and adding $c$, we have 
\[f^{n+1}(0)=(f^{n}(0))^2+c> (f^{n-1}(0))^2+c=f^n(0)>0.\] 
So by induction, we have $f^{n+1}(0)> f^n(0)>0$ for $n\geq 2$. Combining this with the fact that $f^2(0)>a$, we have \[\dots, f^{n+1}(0)> f^n(0)> \dots>f^2(0)> a.\] Therefore, $f^n(0)-a>0$ for all $n\geq 2$.

Now let $a\in (0,\gamma)$. One can check \[f^2(0)-a=a^4-2a^3+2a^2-2a<0\] for $a\in (0,\gamma)$. Then we note $0\leq (a-1)^2=a^2-2a+1$, so $a\leq a^2-a+1=-c$. So we have \[f^2(0)<a\leq -c.\] We claim \[c\leq f^n(0)\leq f^2(0)\] for all $n\geq 1$. Note, $f(0)=c\leq c^2+c=f^2(0)$, so the claim holds for $n=1$ and $n=2$. Now suppose $c\leq f^n(0)\leq f^2(0)$ for some $n\geq 2$. Then we have $c\leq f^n(0)\leq -c$, which implies \[0\leq (f^n(0))^2\leq c^2,\]  and adding $c$, we have \[c\leq (f^n(0))^2+c\leq c^2+c.\] Hence, $c\leq f^{n+1}(0)\leq f^2(0)$, and the claim follows by induction on $n$. Now since $f^n(0)\leq f^2(0)<a$ we have $f^n(0)-a<0$ for each $n\geq 1$, as desired.
\end{proof}

\begin{remark} If $a\in \left(\frac{1}{2}-\frac{\sqrt{5}}{2}, 0\right)\cup \left(\gamma, \frac{1}{2}+\frac{\sqrt{5}}{2}\right)$, the sequence $f^n(0)-a$ takes on both positive and negative values. Note, $\frac{1}{2}\pm \frac{\sqrt{5}}{2}$ are the non-zero roots of $f^3(0)-f^2(0)$.
\end{remark}

\section{Proofs of Main Theorems}\label{sec:proofs}

We are now ready to prove Theorem~\ref{thm:main1} and Theorem~\ref{thm:main2}. Our proofs rely on the following lemma.

\begin{lemma}\cite[Lemma 7.2]{BGJT_25PCFper}\label{lemma:surjective} Let $f(x)=x^2+c\in K[x]$, where $K$ is a field of characteristic not equal to $2$. Define
\begin{equation*}
D_i := \begin{cases}
a-c & \text{if } i=1\\
f^i(0)-a & \text{if } i\geq 2
\end{cases}.
\end{equation*}
Then $\rho: \Gal(\bar{K}/K)\rightarrow \Aut(T)$ is surjective if and only if for all $i\geq 1$, $D_i$ is not a square in $K(\sqrt{D_1},\dots \sqrt{D_{i-1}})$ (where for $i=1$, this means $D_1$ is not a square in $K$).
\end{lemma}

\subsection{Proof of Theorem~\ref{thm:main1}}

Let $a\in \bQ$, $c=-a-a^2$, and $f(x)=x^2+c$. We fix the following notation for this section: write $a=\frac{r}{s}$ with $r,s\in\bZ$, $\gcd(r,s)=1$, and $s>0$ and write $f^n(0)-a=\frac{r_n}{s_n}$ with $r_n,s_n\in\bZ$, $\gcd(r_n,s_n)=1$, and $s_n>0$.

\begin{proposition}\label{prop:family1mod3} 
If $r \not\equiv 0 \mod 3$, then $r_n \equiv 1\mod 3$ for all $n\geq 2$.
\end{proposition}

\begin{proof}
In the proof of Proposition~\ref{prop:family1primes}, we saw $s_n=s^{2^n}$ and $r_{n+1}=r_n^2+2r_n r s^{2^n-1}-2rs^{2^{n+1}-1}$.

First suppose $r\equiv 1\mod 3$. We consider the three cases $s\equiv 0,1,-1\mod 3$.  
\begin{itemize}
\item If $s\equiv 0\mod 3$, then $r_1\equiv -r^2-2rs\equiv -1^2-0 \equiv -1\mod 3$ and 
\[r_2\equiv r_1^2+2r_1 r s^{2^n-1}-2rs^{2^{n+1}-1}\equiv (-1)^2+0-0 \equiv 1\mod 3.\] 
It follows by induction for $n\geq 2$,
\[r_{n+1}\equiv r_n^2+2r_n r s^{2^n-1}-2rs^{2^{n+1}-1}\equiv 1^2+0-0\equiv 1 \mod 3.\]
\item If $s\equiv 1\mod 3$, we have $r_1\equiv  -r^2-2rs\equiv -1^2-2\equiv 0\mod 3$ and 
\[r_2\equiv r_1^2+2r_1 r s^{2^n-1}-2rs^{2^{n+1}-1}\equiv 0^2+0 -2 \equiv 1\mod 3.\] 
It follows by induction for $n\geq 2$,
\[r_{n+1}\equiv r_n^2+2r_n r s^{2^n-1}-2rs^{2^{n+1}-1}\equiv 1^2+2-2 \equiv 1 \mod 3.\]
\item If $s\equiv -1\mod 3$, we have $r_1\equiv  -r^2-2rs\equiv -1^2+2\equiv 1 \mod 3$.
It follows by induction for $n\geq 1$,
\[r_{n+1}\equiv r_n^2+2r_n r s^{2^n-1}-2rs^{2^{n+1}-1}\equiv 1+2\cdot(-1)-2\cdot(-1) \equiv 1 \mod 3.\]
\end{itemize}

Now suppose $r_0\equiv -1 \mod 3$.  We again consider the three cases $s\equiv 0,1,-1\mod 3$.  
\begin{itemize}
\item If $s\equiv 0\mod 3$, then $r_1\equiv -r^2-2rs\equiv -(-1)^2-0 \equiv -1\mod 3$ and 
\[r_2\equiv r_1^2+2r_1 r s^{2^n-1}-2rs^{2^{n+1}-1}\equiv (-1)^2+0-0 \equiv 1\mod 3.\] 
It follows by induction for $n\geq 2$,
\[r_{n+1}\equiv r_n^2+2r_n r s^{2^n-1}-2rs^{2^{n+1}-1}\equiv 1^2+0-0\equiv 1 \mod 3.\]
\item If $s\equiv 1\mod 3$, we have $r_1\equiv  -r^2-2rs\equiv -(-1)^2-2\cdot(-1)\equiv 1\mod 3$.
It follows by induction for $n\geq 2$,
\[r_{n+1}\equiv r_n^2+2r_n r s^{2^n-1}-2rs^{2^{n+1}-1}\equiv 1^2+2\cdot(-1)-2\cdot(-1) \equiv 1 \mod 3.\]
\item If $s\equiv -1\mod 3$, we have $r_1\equiv  -r^2-2rs\equiv -(-1)^2-2\equiv 0 \mod 3$ and 
\[r_2\equiv r_1^2+2r_1 r s^{2^n-1}-2rs^{2^{n+1}-1}\equiv 0^2+0-2\cdot(-1)\cdot(-1) \equiv 1\mod 3.\] 
It follows by induction for $n\geq 2$,
\[r_{n+1}\equiv r_n^2+2r_n r s^{2^n-1}-2rs^{2^{n+1}-1}\equiv 1^2+2\cdot(-1)\cdot(-1)-2\cdot(-1)\cdot(-1)\equiv 1 \mod 3.\]
\end{itemize}

\end{proof}

\begin{proposition}\label{prop:family1mod4}
If $r$ is odd, then $r_n\equiv 1\mod 4$ for $n\geq 2$. 
\end{proposition}

\begin{proof}
By Proposition~\ref{prop:family1primes}, we saw $r_{n+1}=r_n^2+2r_n r s^{2^n-1}-2rs^{2^{n+1}-1}$. So
\[r_1 = -r^2-2rs\equiv \pm 1 \mod 4.\]
Now assume $r_n$ is odd for some $n\geq 1$, then $4$ divides $2r_nrs^{2^{n}-1}-2rs^{2^{n+1}-1} = 2(r_nrs^{2^{n}-1}-rs^{2^{n+1}-1})$, as the expression inside the parentheses is even if $s$ is even or is a difference of odd numbers and hence is even otherwise.
Then for $n\geq 1$, \[r_{n+1}=r_{n}^2+2r_nrs^{2^{n}-1}-2rs^{2^{n+1}-1}\equiv r_{n}^2 \equiv 1\mod 4.\]
\end{proof}

\begin{proof}[Proof of Theorem~\ref{thm:main1}]
We have assumed $D_1=a-c$ is not a square in $\mathbb{Q}$, so by Lemma~\ref{lemma:surjective}, it suffices to show $D_i=f^i(0)-a = \frac{r_i}{s^{2^i}}$ is not a square in $\mathbb{Q}(\sqrt{D_1},\dots \sqrt{D_{i-1}})$ for $i\geq 2$. 

By Corollary~\ref{cor:r_nfam1} and Proposition~\ref{prop:family1signs}, we can write $r_1=\pm2^{e_1}|r|t_1$ and $r_i=(-1)^\delta 2^e|r|t_i$ for $i\geq 2$, where for all $i\geq 1$, $t_i>0$, $\gcd(t_i,r)=1$, $2\nmid t_i$, and $\gcd(t_i,t_j)=1$ for $i>j\geq 1$. Thus, it suffices to show $t_i$ is not a square in $\mathbb{Z}$ for $i\geq 2$. 

We have seen in Proposition~\ref{prop:family1mod3}, that if $r\not \equiv 0 \mod 3$, then $r_i\equiv (-1)^\delta 2^e|r|t_i\equiv 1\mod 3$ for $i\geq 2$. Hence if $(-1)^\delta 2^e |r|\equiv 2\mod 3$, then $t_i\equiv 2\mod 3$, and hence $t_i$ is not a square in $\mathbb{Z}$, proving (\ref{main1_cond1}).

Similarly, working modulo $4$, if $r$ is odd, we have $r_i\equiv (-1)^\delta 2^e |r|t_i \equiv 1\mod 4$ for $i\geq 2$ by Proposition~\ref{prop:family1mod4}. Hence, if $(-1)^\delta 2^e |r|\equiv 3\mod 4$, then $t_i\equiv 3\mod 4$ and hence $t_i$ is not a square in $\mathbb{Z}$, proving (\ref{main1_cond2}).

Finally,  suppose $(-1)^\delta 2^e |r| \mod q$ is not a quadratic residue for some prime $q$ dividing $s$. Then since $r_i\equiv (r_{i-1})^2 \mod q$ for $i\geq 2$ is a non-zero quadratic residue and $r_i=(-1)^\delta 2^e |r|t_i$, it follows that $t_i$ is a quadratic non-residue modulo $q$ and hence, $t_i$ is not a square, proving (\ref{main1_cond3}).
\end{proof}

\subsection{Proof of Theorem~\ref{thm:main2}}

Let $a\in \bQ$, $c=-1+a-a^2$, and $f(x)=x^2+c$. We fix the following notation for this section: write $a=\frac{r}{s}$ with $r,s\in \bZ$, $\gcd(r,s)=1$, and $s>0$ and write $f^n(0)-a=\frac{r_n}{s_n}$ with $r_n,s_n\in\bZ$, $\gcd(r_n,s_n)=1$, and $s_n>0$.

\begin{proof}[Proof of Theorem~\ref{thm:main2}] Let $D_1,D_2,\dots$ be defined as in Lemma~\ref{lemma:surjective}. We need to show $D_i$ is not a square in $\mathbb{Q}(\sqrt{D_1},\dots, \sqrt{D_{i-1}})$ for all $i$. This holds for $i=1$ by hypothesis. 

Note, in each of the cases we are considering, we have $s>r>0$, so $0<a<1$. Hence by Proposition~\ref{prop:family2signs}, $f^n(0)-a<0$ for all $n\geq 2$. 

First suppose $r=1$ and $s$ is even, then by Proposition~\ref{prop:family2primes}, condition~\ref{item:repeatedprimes_fam2}, we have $\gcd(r_i,r_j)=1$ for $i\neq j$. So it suffices to show $|r_i|=-r_i$ is not a square in $\bZ$ for $i\geq 2$. By the proof of Proposition~\ref{prop:family2primes}, we have 
$r_{n+1}=r_n^2+2r_nrs^{2^n-1}-s^{2^{n+1}}$. Note, $r_1$ is odd and if $r_n$ is odd for some $n\geq 1$, then 
\[r_{n+1}\equiv r_n^2\equiv 1\mod 4,\]
hence $r_n\equiv 1\mod 4$ for all $n\geq 2$ by induction on $n$.
Thus, $|r_i|=-r_i\equiv -1\mod 4$ for $i\geq 2$, which implies $|r_i|$ is not a square, so (\ref{main2_cond1}) holds.

Now let $r=2$. By Proposition~\ref{prop:family2primes} and Proposition~\ref{prop:family2signs}, we have
\begin{equation*}
r_i=\begin{cases}
-t_i & \text{if $i$ is odd}\\
-2^2t_i & \text{if $i$ is even}
\end{cases}
\end{equation*}
where $t_i>0$, $2\nmid t_i$, and $\gcd(t_i,t_j)=1$ for $i>j\geq 1$. It suffices to show $t_i$ is not a square in $\bZ$ for $i\geq 2$.

Suppose $s\equiv 1\mod 3$. We show $r_n\equiv 1\mod 3$ for $n\geq 1$. We have \[r_1\equiv -r^2-s^2\equiv -2^2-1^2\equiv 1\mod 3.\]
Now suppose $r_n\equiv 1\mod 3$ for some $n\geq 1$, then
\begin{align*}
r_{n+1}&\equiv r_n^2+2r_nrs^{2^n-1}-s^{2^{n+1}}\mod 3\\
&\equiv 1^2+2\cdot 1\cdot 2\cdot 1^{2^n-1}-1^{2^{n+1}}\mod 3\\
&\equiv 1+2\cdot 2-1\mod 3\\
&\equiv 1 \mod 3,
\end{align*}
hence, $r_n\equiv 1\mod 3$ for all $n\geq 1$ as desired.
Then for all $i\geq 2$, $t_i\equiv -r_i\equiv -1\mod 3$ is not a square modulo $3$ and hence is not a square in $\mathbb{Z}$, proving (\ref{main2_cond2}).

Finally, suppose there is some prime $q$ such that $q\mid s$ and $q\equiv 3\mod 4$. Then for $n\geq 1$, we have
\[r_{n+1}=r_n^2+2r_nrs^{2^n-1}-s^{2^{n+1}}\equiv r_n^2 \mod q.\]
Hence, for $i\geq 2$, $r_i$ is a square modulo $q$. Then for odd $i$, we have $-t_i$ is a square modulo $q$ and since $-1$ is not a square mod $q$, $t_i$ is also not a square mod $q$ and hence is not a square in $\mathbb{Z}$. Similarly, for even $i$, $-2^2t_i$ is a square modulo $q$, so $-2^2t_i$ is a square modulo $q$ and again we conclude $t_i$ is not a square in $\mathbb{Z}$, and (\ref{main2_cond3}) holds. 
\end{proof}

\newcommand{\etalchar}[1]{$^{#1}$}
\providecommand{\bysame}{\leavevmode\hbox to3em{\hrulefill}\thinspace}
\providecommand{\MR}{\relax\ifhmode\unskip\space\fi MR }
% \MRhref is called by the amsart/book/proc definition of \MR.
\providecommand{\MRhref}[2]{%
  \href{http://www.ams.org/mathscinet-getitem?mr=#1}{#2}
}
\providecommand{\href}[2]{#2}

\end{document}